\documentclass[12pt,english]{amsart}
\usepackage{amsfonts,amsmath,amsthm,amscd,amssymb,latexsym,amsbsy,pb-diagram,cite,caption,url}
\usepackage[mathscr]{eucal}
\usepackage[T2A]{fontenc}
\usepackage[cp1251]{inputenc}
\usepackage[english]{babel}

\usepackage{tikz} 

\newtheorem{theorem}{Theorem}[section]
\newtheorem{corollary}[theorem]{Corollary}
\newtheorem{lemma}[theorem]{Lemma}

\theoremstyle{definition}
\newtheorem{definition}[theorem]{Definition}
\theoremstyle{remark}
\newtheorem{remark}[theorem]{Remark}
\numberwithin{equation}{section}

\newcommand{\mct}{\mathfrak{T}}

\DeclareMathOperator{\diam}{diam}
\newcommand{\Sp}[1]{\operatorname{Sp}(#1)}
\def\we{\mathrel{\stackrel{\rm w}=}}
\begin{document}

\title[Weak similarities of finite ultrametric spaces]
{Weak similarities of finite\\ ultrametric and semimetric spaces}

\author{Evgeniy Petrov}

\address{Institute of Applied Mathematics and Mechanics of NASU, Dobrovolskogo str. 1, Slovyansk 84100, Ukraine}

\email{eugeniy.petrov@gmail.com}



\thanks{This is a preprint of the work accepted for publication in ``p-Adic Numbers, Ultrametric Analysis and Applications'', \copyright,  copyright 2018, the copyright holder is indicated in the Journal. \url{http://pleiades.online}}

\begin{abstract}
Weak similarities form a special class of mappings between semimetric spaces. Two semimetric spaces $X$ and $Y$ are weakly similar if there exists a weak similarity $\Phi\colon X\to Y$. We find a structural characteristic of finite ultrametric spaces for which the isomorphism of its representing trees implies a weak similarity of the spaces. We also find conditions under which the Hasse diagrams of balleans of finite semimetric spaces are isomorphic.
\end{abstract}

\maketitle
\section{Introduction}

The notion of weak similarity of semimetric spaces was introduced in~\cite{DP2},  where also some properties of these mappings were studied. It was shown that the weak similarities between geodesic spaces are usual similarities and every weak similarity  between compact ultrametric spaces with equal distance sets is an isometry. Some conditions under which weak similarities are homeomorphisms or uniform equivalences were also found.

A class $\mathfrak R$ of finite ultrametric spaces which are ``as rigid as possible'' was considered in~\cite{DPT(Howrigid)}. It was proved that a metric space that is weakly similar to a space from the class $\mathfrak R$ also belongs to $\mathfrak R$. It was shown that a weak similarity of two spaces $X, Y \in \mathfrak R$ is equivalent to the isomorphism of representing trees of these spaces as well as to the equality $|X|=|Y|$.

The concept of weak similarity is closely connected to such concepts as the ordinal scaling, the multidimensional scaling and the ranking which have a lot of applications, see e.g.~\cite{AWC07, BG05, JN11, LK14, K64, QY04, RF06, Sh62, Sh66, WJJ13}. The presence of these relationships and, thus, the existence of potential applications, makes the study of weak similarities much more important.

In this paper we continue to study weak similarities between finite semimetric and in particular between finite ultrametric spaces. In Theorem~\ref{t2} we find a characterization of finite ultrametric spaces in terms of its representing trees for which the isomorphism of representing trees implies a weak similarity of the spaces. In Theorem~\ref{t3} we study the same question with an additional condition of injective labeling of representing trees. In Section~\ref{s3} we prove that a weak similarity of two finite semimetric spaces implies the isomorphism of Hasse diagrams of these space, see Theorem~\ref{t4}. We also show that Hasse diagrams of two finite semimetric spaces are isomorphic as directed graphs if and
only if there exists a bijective ball-preserving mapping between these spaces, see Theorem~\ref{t5}.

Recall some definitions from the theory of metric spaces and the graph theory.
An \textit{ultrametric} on a set $X$ is a function $d\colon X\times X\rightarrow \mathbb{R}^+$, $\mathbb R^+ = [0,\infty)$, such that for all $x,y,z \in X$:
\begin{itemize}
\item [(i)] $d(x,y)=d(y,x)$,
\item [(ii)] $(d(x,y)=0)\Leftrightarrow (x=y)$,
\item [(iii)] $d(x,y)\leq \max \{d(x,z),d(z,y)\}$.
\end{itemize}
Inequality (iii)  is often called the {\it strong triangle inequality}.
The pair $(X,d)$ is called an \emph{ultrametric space.} If condition (iii) is omitted, then $(X,d)$ is a \emph{semimetric space}, see for example~\cite{Bl}. It is clear that every ultrametric space is semimetric space.
The \emph{spectrum} of an ultrametric space $(X,d)$ is the set  $$\operatorname{Sp}(X)=\{d(x,y)\colon x,y \in  X\}.$$
For simplicity we will always assume that $X\cap \Sp{X}=\varnothing$. The quantity
$$
\diam X=\sup\{d(x,y)\colon x,y\in X\}.
$$
is the \emph{diameter} of the space $(X,d)$.

Recall that a \textit{graph} is a pair $(V,E)$ consisting of a nonempty set $V$ and a (probably empty) set $E$  elements of which are unordered pairs of different points from $V$. For a graph $G=(V,E)$, the sets $V=V(G)$ and $E=E(G)$ are called \textit{the set of vertices} and \textit{the set of edges}, respectively. Recall that a \emph{path} is a nonempty graph $P=(V,E)$ of the form
$$
V=\{x_0,x_1,...,x_k\}, \quad E=\{\{x_0,x_1\},...,\{x_{k-1},x_k\}\},
$$
where all $x_i$ are distinct.
A connected graph without cycles is called a \emph{tree}. A tree $T$ may have a distinguished vertex called the \emph{root}; in this case $T$ is called a \emph{rooted tree}.  An \emph{$n$-ary tree}~--- is a rooted tree, such that the degree of each of its vertices is at most $n+1$. A rooted tree is \emph{strictly binary} if every internal node has exactly two children.
By $L_T$ we denote the set of leaves of the tree $T$. Generally we  follow terminology used in~\cite{BM}.

\begin{definition}\label{def3.1}
Let $k\geqslant 2$. A nonempty graph $G$ is called \emph{complete $k$-partite} if its vertices can be divided into $k$ disjoint nonempty sets $X_1,...,X_k$ so that there are no edges joining the vertices of the same set $X_i$ and any two vertices from different $X_i,X_j$, $1\leqslant i,j \leqslant k$ are adjacent. In this case we write $G=G[X_1,...,X_k]$.
\end{definition}
We shall say that $G$ is a {\it complete multipartite graph} if there exists
 $k \geqslant 2$ such that $G$ is complete $k$-partite, cf. \cite{Di}.

\begin{definition}[\!\cite{DDP(P-adic)}]\label{d2}
Let $(X,d)$ be a finite ultrametric space. Define the graph $G_X^d$ as follows $V(G_X^d)=X$ and
$$
(\{u,v\}\in E(G_X^d))\Leftrightarrow(d(u,v)=\diam X).
$$
We call $G_X^d$ the \emph{diametrical graph} of $X$.
\end{definition}

\begin{theorem}[\!\cite{DDP(P-adic)}]\label{t13}
Let $(X,d)$ be a finite ultrametric space, $|X|\geqslant 2$. Then $G_X^d$ is complete multipartite.
\end{theorem}

With every finite ultrametric space $(X, d)$, we can associate  a labeled rooted $n$-ary tree $T_X$ by the following rule (see~\cite{PD(UMB)}). If $X=\{x\}$ is a one-point set, then $T_X$ is the rooted tree consisting of one node $X$ labeled by $0$. Let $|X|\geqslant 2$.
According to Theorem~\ref{t13} we have $G^d_X = G^d_X[X_1,...,X_k]$.
In this case the root $X$ of the tree $T_X$ is labeled by $\diam X$ and, moreover, $T_X$ has $k$ nodes $X_1,...,X_k$ of the first level with the labels
\begin{equation}\label{e2.7}
l(X_i)=
\diam X_i, \quad i = 1,...,k.
\end{equation}
The nodes of the first level indicated by $0$ are leaves, and those indicated by strictly positive numbers are internal nodes of the tree $T_X$. If the first level has no internal nodes, then the tree $T_X$ is constructed. Otherwise, by repeating the above-described procedure with $X_i$ corresponding to the internal nodes of the first level, we obtain the nodes of the second level, etc. Since $|X|$ is finite, all vertices on some level will be leaves and the construction of $T_X$ is completed.

The above-constructed labeled tree $T_X$ is called the \emph{representing tree} of the space $(X, d)$. To underline that $l_X(x)$ is a labeling function of the representing tree $T_X$ we shall write $(T_X,l_X)$. The rooted tree $T_X$ without the labels we will denote by $\overline{T}_X$.

Note that the correspondence between trees and ultrametric spaces is well known, cf.~\cite{GV, GNS00,GurVyal(2012),H04}.

\begin{definition}
Let $T_1$ and $T_2$ be rooted trees with the roots $v_1$ and $v_2$ respectively. A bijective function $\Psi\colon V(T_1)\to V(T_2)$ is an isomorphism of $T_1$ and $T_2$ if
$$
(\{x,y\}\in E(T_1))\Leftrightarrow(\{\Psi(x),\Psi(y)\}\in E(T_2))
$$
for all $x,y \in V(T_1)$ and $\Psi(v_1)=v_2$. If there exists an isomorphism of rooted trees $T_1$ and $T_2$, then we will write $T_1\simeq T_2$.
\end{definition}

We shall say that trees $(T_X,l_X)$ and $(T_Y,l_Y)$ are isomorphic as labeled rooted trees if  $\overline{T}_X \simeq \overline{T}_Y$ with isomorphism $\Psi$ and $l_X(x)=l_Y(\Psi(x))$.

\section{Weak similarities of finite ultrametric spaces}

Recall that a mapping $\Phi$ from a metric space $(X, d)$ to a metric space $(Y, \rho)$ is a \emph{similarity} if there is $\lambda>0$ such that
$$
\lambda(d(x,y)) = \rho(\Phi(x),\Phi(y))
$$
for all $x$, $y \in X$.

\begin{definition}\label{d4.5}
Let  $(X,d)$ and $(Y,\rho)$ be semimetric spaces. A bijective mapping $\Phi\colon X\to Y$ is a \emph{weak similarity} if there exists a strictly increasing bijection $f\colon \Sp{X}\to \Sp{Y}$ such that the equality
\begin{equation}\label{e4.5}
f(d(x,y))=\rho(\Phi(x),\Phi(y))
\end{equation}
holds for all $x$, $y\in X$. The function $f$ is said to be a \emph{scaling function} of $\Phi$. If $\Phi\colon X\to Y$ is a weak similarity, we write $X \we Y$ and say that $X$ and  $Y$  are \emph{weakly similar}. The pair $(f,\Phi)$ is called a \emph{realization} of $X\we Y$.
\end{definition}

In~\cite{DP2} the notion of weak similarity was introduced in a slightly different form.

\begin{remark}\label{r2}
The pair $(f,\Phi)$ is a realization of $(X,d)\we (Y,\rho)$ if and only if $(X,f\circ d)$ and $(Y,\rho)$ are isometric with the isometry $\Phi$.
\end{remark}

\begin{theorem}[\!\cite{DPT(Howrigid)}] \label{t2.6}
Let $(X, d)$ and $(Y, \rho)$ be finite nonempty ultrametric spaces. Then $(X, d)$ and $(Y, \rho)$ are isometric if and only if $T_X$ and $T_Y$ are isomorphic as labeled rooted trees.
\end{theorem}

It is easy to see that for finite ultrametric spaces $X$ and $Y$ the isomorphism of $T_X$ and $T_Y$ does not imply that $X$ and $Y$ are weakly similar. But the converse is true.
\begin{lemma}\label{p2.5}
Let $(X,d)$ and $(Y,\rho)$ be finite ultrametric spaces. If $X\we Y$ with realization $(f,\Phi)$ then the following statements hold.
\begin{itemize}
  \item [(i)] The representing tree of the ultrametric space $(X,f\circ d)$ is isomorphic to $T_Y$.
  \item [(ii)] $\overline{T}_X \simeq \overline{T}_Y$.
\end{itemize}
\end{lemma}
\begin{proof}
Statement (i) follows from Remark~\ref{r2} and Theorem~\ref{t2.6}.
Let $X'=(X,f\circ d)$. The isomorphism of labeled trees $T_Y$ and $T_{X'}$ implies $\overline{T}_Y\simeq \overline{T}_{X'}$. It is easy to see that $\overline{T}_{X'}\simeq \overline{T}_X$ which establish statement (ii).
\end{proof}

In what follows we need the following lemma.
\begin{lemma}[\!~\cite{PD(UMB)}]\label{l2} Let $(X, d)$ be a finite ultrametric space and let $x_1$ and $x_2$ be two different leaves of the tree $T_X$. If
$(x_1, v_1, . . . , v_n, x_2)$ is the path joining the leaves $x_1$ and $x_2$ in $T_X$, then
\begin{equation}\label{e1}
d(x_1, x_2) = \max\limits_{1\leqslant i \leqslant n} l({v}_i).
\end{equation}
\end{lemma}

A class $\mathfrak R$ of finite ultrametric spaces which are ``as rigid as possible'' was considered in~\cite{DPT(Howrigid)}. It was established that the representing tree $T_X$ of the space $X\in \mathfrak R$ is strictly binary with exactly one inner node at each level except the last level. Denote by $\tilde{\mathfrak R}$ the class of finite ultrametric spaces $X$ for which $T_X$ has exactly one inner node at each level expect the last level. It is clear that $\mathfrak R$ is a proper subclass of $\tilde{\mathfrak R}$.

\begin{figure}[h]
\begin{tikzpicture}[scale=0.7]

\draw (-1,8) node [right] {$T_X$};

 \foreach \i in {(1,8),(0,6),(1,6),(2,6),(1,4),(1+2/3,4),(1+4/3,4),(3,4),(2,2),(3,2),(4,2),(3,0),(4,0),(5,0)}
  \fill[black] \i circle (2pt);

\draw (1,8)  node [right] {$l_0$}  -- (0,6);
\draw (1,8)  -- (1,6);
\draw (1,8)  -- (2,6) node [right] {$l_{1}$};

\draw (2,6)  -- (1,4);
\draw (2,6)  -- (1+2/3,4);
\draw (2,6)  -- (1+4/3,4);
\draw[dotted] (2,6)  -- (3,4) node [right] {$l_{k-1}$};

\draw (3,4)  -- (2,2);
\draw (3,4)  -- (3,2);
\draw (3,4)  -- (4,2) node [right] {$l_{k}$};

\draw (4,2)  -- (3,0);
\draw (4,2)  -- (4,0);
\draw (4,2)  -- (5,0);

\foreach \j in {6}
{

\draw (-1+\j,8) node [right] {$T_Y$};
 \foreach \i in {(1+\j,8),(0+\j,6),(1+\j,6),(2+\j,6),(1+\j,4),(1+2/3+\j,4),(1+4/3+\j,4),(3+\j,4),(2+\j,2),(3+\j,2),(4+\j,2),(3+\j,0),(4+\j,0),(5+\j,0)}
  \fill[black] \i circle (2pt);

\draw (1+\j,8)  node [right] {$\tilde{l}_0$}  -- (0+\j,6);
\draw (1+\j,8)  -- (1+\j,6);
\draw (1+\j,8)  -- (2+\j,6) node [right] {$\tilde{l}_{1}$};

\draw (2+\j,6)  -- (1+\j,4);
\draw (2+\j,6)  -- (1+2/3+\j,4);
\draw (2+\j,6)  -- (1+4/3+\j,4);
\draw[dotted] (2+\j,6)  -- (3+\j,4) node [right] {$\tilde{l}_{k-1}$};

\draw (3+\j,4)  -- (2+\j,2);
\draw (3+\j,4)  -- (3+\j,2);
\draw (3+\j,4)  -- (4+\j,2) node [right] {$\tilde{l}_{k}$};

\draw (4+\j,2)  -- (3+\j,0);
\draw (4+\j,2)  -- (4+\j,0);
\draw (4+\j,2)  -- (5+\j,0);
}
\end{tikzpicture}
\caption{$X\we Y$ with $\overline{T}_X\simeq \overline{T}_Y$}\label{f4}
\label{fig1}
\end{figure}
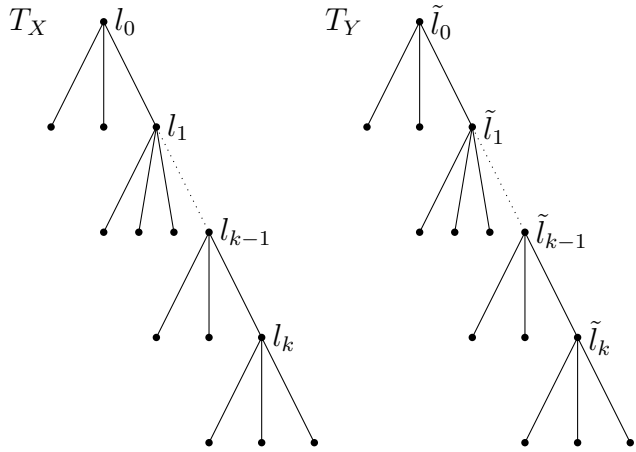

The next theorem gives a description of finite ultrametric spaces for which the isomorphism of representing  trees implies the weak similarity of the spaces.
\begin{theorem}\label{t2}
Let $X$ be a finite ultrametric space. Then the following statements are equivalent.
\begin{itemize}
  \item [(i)] The implication $(\overline{T}_X\simeq \overline{T}_Y) \Rightarrow (X\we Y)$ holds for every finite ultrametric space $Y$.
  \item [(ii)] $X\in \tilde{\mathfrak{R}}$.
\end{itemize}
\end{theorem}
\begin{proof}
(i)$\Rightarrow$(ii). Let condition (i) hold and let $\Psi\colon V(\overline{T}_X)\to V(\overline{T}_Y)$ be an isomorphism of $\overline{T}_X$ and $\overline{T}_Y$.  Suppose $X \notin \tilde{\mathfrak{R}}$. Consequently there exist at least two inner nodes $x_1, x_2 \in V(T_X)$ at one and the same level.
For every inner node $y\in V(T_Y)$ define a labeling by the following way
$$
l_Y(y)=
\begin{cases}
l_X(x), &\text{if } \, \ y=\Psi(x) \ \text{ and } \ x\neq x_2,\\
l_X(x_1), &\text{if } \, \ y=\Psi(x_2) \ \text{ and } \ l_X(x_1)\neq l_X(x_2),\\
a, &\text{if } \, \ y=\Psi(x_2) \ \text{ and } \ l_X(x_1)= l_X(x_2),
\end{cases}
$$
where $a$ is an appropriate positive real number such that $a\neq l_X(x_1)$ (it is always possible according to construction of representing tree).

It is clear that $X$ and $Y$ are not weakly similar because  $|\Sp{X}|\neq |\Sp{Y}|$ (see Definition~\ref{d4.5}).

(ii)$\Rightarrow$(i). Let $X\in \tilde{\mathfrak{R}}$ and let $Y$ be a finite ultrametric space such that $\overline{T}_X\simeq \overline{T}_Y$ with the isomorphism $\Psi$.
Since $X\in \tilde{\mathfrak{R}}$ all labels of inner nodes of $T_X$ are different. Analogically, $\overline{T}_X\simeq \overline{T}_Y$ implies $Y\in \tilde{\mathfrak{R}}$, so that the labels of inner nodes of $T_Y$ are also different. Moreover, the numbers of inner nodes of $T_X$ and $T_Y$ coincide since $T_X$ and $T_Y$ are isomorphic. Let $l_0,l_1,...,l_k$ and $\tilde{l}_0,\tilde{l}_1,...,\tilde{l}_k$ be the labels of inner nodes of $T_X$ and $T_Y$, respectively,  decreasingly ordered, see Figure~\ref{fig1}. It is clear that the labels $l_i$ and $\tilde{l}_i$ are on the same level. Define $f$ as follows
\begin{equation}\label{e18}
f(l_i)=\tilde{l}_i, \ \, i=0,...,k, \ \, f(0)=0.
\end{equation}
Define $\Phi = \Psi|_{L_{T_X}}$. Let us prove that $X\we Y$ with the realization $(f,\Phi)$. Without loss of generality,  suppose that $x$ is a direct successor of some inner node with label $l_i$ and $y$ is a direct successor of some inner node with the label $l_j$ where $i\leqslant j$. According to Lemma~\ref{l2} and to the construction of representing trees $T_X$ and $T_Y$ we have $d(x,y)=l_i$ and $\rho(\Phi(x), \Phi(y))=\tilde{l}_i$ where $d$ and $\rho$ are the ultrametrics on $X$ and $Y$ respectively. Using ~(\ref{e18}) we obtain the equality in~(\ref{e4.5}) which completes the proof.
\end{proof}

Denote by $\mathfrak D$ the class of all finite ultrametric spaces $X$ such that the different internal nodes of $T_X$ have the different labels. In this case we will say that $T_X$ has an \emph{injective labeling}. It is clear that $\mathfrak R$ and $\tilde{\mathfrak  R}$ are subclasses of $\mathfrak D$. A question arises whether there exist finite ultrametric spaces $X$, $Y\in \mathfrak D$ which do not belong to the class $\tilde{\mathfrak  R}$ and for which the isomorphism of $\overline{T}_X$ and $\overline{T}_Y$ implies $X\we Y$.

Let us define a rooted tree $T$  with $n$ levels by the following two conditions:

(A) There is only one inner node at the level $k$ of $T$ whenever $k<n-1$.

(B) If $u$ and $v$ are different inner nodes at the level $n-1$ then the numbers of offsprings of $u$ and $v$ are equal.

\begin{figure}[h]
\begin{tikzpicture}[scale=0.6]

\foreach \j in {4}
{
\foreach \i in {(1+\j,9),(0+\j,6), (1/2+\j,6), (1+\j,6), (3/2+\j,6), (2+\j,6)}
  \fill[black] \i circle (2pt);
\draw (1+\j,9)  -- (0+\j,6);
\draw (1+\j,9)  -- (1/2+\j,6);
\draw[dotted] (1+\j,9)  -- (1+\j,6);
\draw[dotted] (1+\j,9)  -- (3/2+\j,6);
\draw (1+\j,9)  -- (2+\j,6);
}

\foreach \j in {3}
{
\foreach \i in {(1+\j,12),(0+\j,9), (1/2+\j,9), (1+\j,9), (3/2+\j,9), (2+\j,9)}
  \fill[black] \i circle (2pt);
\draw (1+\j,12)  -- (0+\j,9);
\draw (1+\j,12)  -- (1/2+\j,9);
\draw[dotted] (1+\j,12)  -- (1+\j,9);
\draw[dotted] (1+\j,12)  -- (3/2+\j,9);
}

\foreach \j in {2}
{
\foreach \i in {(1+\j,15),(0+\j,12), (1/2+\j,12), (1+\j,12), (3/2+\j,12), (2+\j,12)}
  \fill[black] \i circle (2pt);
\draw (1+\j,15)  -- (0+\j,12);
\draw (1+\j,15)  -- (1/2+\j,12);
\draw[dotted] (1+\j,15)  -- (1+\j,12);
\draw[dotted] (1+\j,15)  -- (3/2+\j,12);
\draw (1+\j,15)  -- (2+\j,12);
}

\draw (3,15) node [right] {$r_{0}$};
\draw (4,12) node [right] {$r_{1}$};
\draw[dotted] (4,12)  -- (5,9) node [right] {$r_{k-1}$};
\draw (6,6.2) node [right] {$r_{k}$};

\draw (0,15) node [right] {$T$};
        \draw (6,6)  -- (1,3) node [left] {$\textstyle x_{1}$};
        \draw (6,6)  -- (4,3)  node [left] {$ \textstyle x_{2}$};
\draw[dotted] (6,6)  -- (4+4/3,3);
\draw[dotted] (6,6)  -- (4+8/3,3);
        \draw (6,6)  -- (8,3)  node [left] {$\textstyle x_{m}$};
        \draw (6,6)  -- (9,3);
\draw[dotted] (6,6)  -- (10,3);
\draw[dotted] (6,6)  -- (11,3);
        \draw (6,6)  -- (12,3);

\foreach \i in {(6,6),(9,3),(12,3),(10,3),(11,3)}
  \fill[black] \i circle (2pt);

\foreach \j in {0,3,7}
{
\foreach \i in {(1+\j,3),(0+\j,0), (1/2+\j,0), (1+\j,0), (3/2+\j,0), (2+\j,0)}
  \fill[black] \i circle (2pt);
\draw (1+\j,3)  -- (0+\j,0);
\draw (1+\j,3)  -- (1/2+\j,0);
\draw[dotted] (1+\j,3)  -- (1+\j,0);
\draw[dotted] (1+\j,3)  -- (3/2+\j,0);
\draw (1+\j,3)  -- (2+\j,0);
}
\end{tikzpicture}
\caption{The structure of the tree $T_X$ with $X\in\mct$.}\label{fig2}
\end{figure}
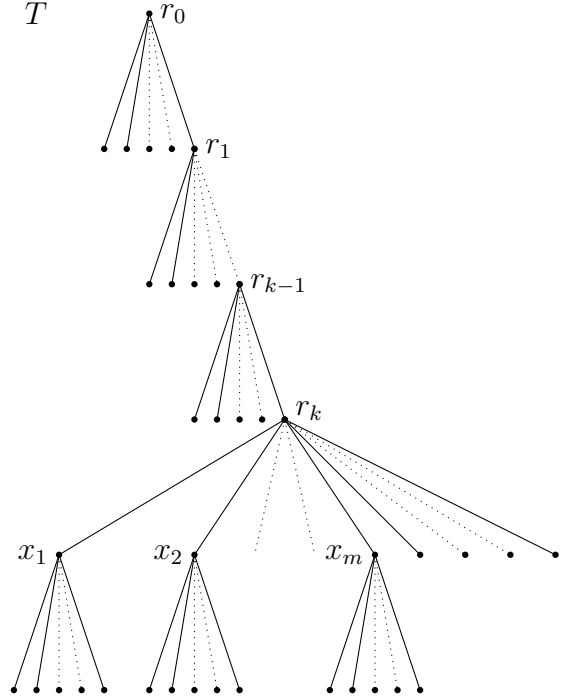

Denote by $\mct$ the class of all finite ultrametric spaces $X$ for which $T_X$ satisfies conditions (A) and (B) (See Figure~\ref{fig2}).

\begin{theorem}\label{t3}
Let $X \in \mathfrak D$ be a finite ultrametric space. Then the following statements are equivalent.
\begin{itemize}
  \item [(i)] The implication $(\overline{T}_X\simeq \overline{T}_Y) \Rightarrow (X\we Y)$ holds for every finite ultrametric space $Y \in \mathfrak D$.
  \item [(ii)] $X \in \mct$.
\end{itemize}
\end{theorem}

\begin{proof}
(i)$\Rightarrow$(ii). Let $T_X$ satisfy condition (i) and have $n$ levels. Suppose (A) does not hold for $T=T_X$. Then $n\geqslant 3$ and there are $Y_1, Y_2 \in \mathfrak D$ such that:
\begin{itemize}
  \item [$(a_1)$] $\overline{T}_{Y_1}\simeq\overline{T}_{X}\simeq\overline{T}_{Y_2}$;
  \item [$(a_2)$] the inequality $l_{Y_1}(u_1)<l_{Y_1}(v_1)$ holds whenever $u_1$ and $v_1$ are inner nodes of $\overline{T}_{Y_1}$ and the level of $v_1$ is strictly less than the level of $u_1$;
  \item [$(a_3)$] $\overline{T}_{Y_2}$ contains inner nodes $u_2$ and $v_2$ such that $l_{Y_2}(v_2)<l_{Y_2}(u_2)$ and the level of $v_2$ is strictly less than the level of $u_2$.
\end{itemize}
It follows from (i) and $(a_1)$, that $Y_1\we Y_2$. Hence $(a_2)$ contradicts $(a_3)$ because $(a_2)$ and $(a_3)$ are invariant under weak similarities.  Thus (i) implies (A) for $T=T_X$.

To prove $B$ we consider two inner nodes $v_1$ and $v_2$ of $\overline{T}_{X}$ on the level $n-1$. Denote by $d_0$ the original ultrametric on $X$. It is easy to find some ultrametric spaces $(X,d_1)$, $(X,d_2)\in \mathfrak D$ such that:
\begin{itemize}
  \item [$(b_1)$] $(X,d_0)$, $(X,d_1)$, $(X,d_2)$ have one and the same representing tree up to labeling;
  \item [$(b_2)$] $\diam(v_i)$ is the least strictly positive number from $\Sp{(X,d_i)}$, $i=1,2$.
\end{itemize}
It follows from (i) and $(b_1)$ that $(X,d_1)$ and $(X,d_2)$ are weakly similar. If $\Phi\colon (X,d_1)\to (X,d_2)$ is a weak similarity, then $(b_2)$ implies the equality $v_2=\Phi(v_1)$.
Consequently the numbers of offsprings of $v_1$ and $v_2$ are equal. Thus (B) follows from (i).

(ii)$\Rightarrow$(i). Let $X \in \mathfrak{T}$. Without loss of generality suppose $|X|\geqslant 2$. To prove (i) we consider an arbitrary $Y\in \mathfrak D$ such that $\overline{T}_{X}\simeq\overline{T}_{Y}$.

Let $V^i(T_X)$ and $V^i(T_Y)$ be the sets of the inner nodes of $(T_X,l_X)$ and $(T_Y,l_Y)$ respectively. From $X,Y\in \mathfrak D$ and $\overline{T}_{X}\simeq\overline{T}_{Y}$ it follows that there is $k\geqslant 1$ such that
$$
k=|\Sp{X}|-1=|\Sp{Y}|-1=|V^i(T_X)|=|V^i(T_Y)|.
$$
Let $V^i(T_X)=(X_1,\ldots,X_k)$, $V^i(T_Y)=(Y_1,\ldots,Y_k)$ and
$$
\Sp{X}\setminus \{0\} = \{r_1,...,r_k\}, \, \ \Sp{Y}\setminus \{0\} = \{t_1,...,t_k\}
$$
such that $l_X (X_i)=r_i$ and $l_Y (Y_i)=t_i$ for $i=1,...,k$ and
$$
r_1>r_2>\ldots r_k >0 \, \ \text{ and } \, \ t_1>t_2>\ldots t_k >0.
$$
Using (ii) and $\overline{T}_{X}\simeq\overline{T}_{Y}$ we see that the function $X_i \mapsto Y_i$, $i=1,\ldots,k$, can be extended to an isomorphism $\Phi\colon \overline{T}_{X}\to\overline{T}_{Y}$, and, in addition, the mapping
$$
X\ni x \mapsto \{x\} \mapsto \Phi(\{x\})=\{y\} \mapsto y \in Y
$$
is a weak similarity. Hence $\overline{T}_{X}\simeq\overline{T}_{Y}$ implies $X\we Y$.
\end{proof}

\section{Weak similarities of finite semimetric spaces}\label{s3}

Let  $(Y,\leqslant_Y)$ be a finite partially ordered set. A \emph{Hasse diagram} of the poset $(Y,\leqslant_Y)$ is a directed graph with the set of vertices $Y$ and the set of arcs (directed edges) $A_Y\subseteq Y\times Y$ such that the pair $\langle v_1,v_2 \rangle$ belongs to $A_Y$ if and only if $v_1\leqslant_Y v_2$, $v_1\neq v_2$ and the implication
$$
(v_1\leqslant_Y w\leqslant_Y v_2)\Rightarrow (v_1=w \vee v_2=w)
$$
holds for every $w\in Y$.

Two digraphs $(Y,A_Y)$ и $(X,A_X)$ are isomorphic, if there exists a bijection $F\colon X\to Y$ such that
$$
  (\langle x,y\rangle\in A_X)\Leftrightarrow(\langle F(x),F(y)\rangle \in A_Y)
$$
for all $x,y \in X$.
In this case $F$ is an isomorphism of the digraphs $(Y,A_Y)$ and $(X,A_X)$.
Every rooted tree $T$ can be considered as a digraph $(V(T),A_T)$, under assumption
$$
  (\langle u,v \rangle \in A_T) \Leftrightarrow  (u \text{ is a direct successor of } v).
$$

We will denote by $\mathcal H(Y)$ the Hasse diagram of a partially ordered set $(Y,\leqslant_Y)$.

Let $(X, d)$ be a semimetric space. Recall that a subset $B$ of $X$ is a \emph{ball} in $(X,d)$ if there exists $r \geq 0$ and $t \in X$ such that
$$
B=\{x\in X\colon d(x,t)\leq r\}.
$$
In this case we write $B=B_r(t)$. By $\mathbf{B}_X$ we denote the \emph{ballean} of $(X,d)$, i.e., the set of all balls in $(X, d)$.

It is clear that for every semimetric space $X$ the set $\mathbf{B}_X$ can be considered as a poset $(\mathbf{B}_X,\subseteq)$ with the partial order defined by the relation of inclusion $\subseteq$.

\begin{theorem}\label{t4}
Let $X$ and $Y$ be finite nonempty semimetric spaces. If $X\we Y$, then $\mathcal H(\mathbf{B}_X)$ and $\mathcal H(\mathbf{B}_Y)$ are isomorphic as directed graphs.
\end{theorem}
\begin{proof}
Let $d$ and $\rho$ be semimetrics on $X$ and $Y$ respectively and let $(f,\Phi)$ be a realization of $X\we Y$. We claim that the  posets $\mathcal H(\mathbf{B}_X)$ and $\mathcal H(\mathbf{B}_Y)$ are isomorphic with the isomorphism
\begin{equation}\label{iso}
\mathbf{B}_X \ni B \mapsto \Phi(B) \in \mathbf{B}_Y.
\end{equation}

Let us prove first that if  $B\in \mathbf{B}_X$, then $\Phi(B)\in \mathbf{B}_Y$.
Let $B=B_r(t)\in \mathbf{B}_X$ and let $r_0=\max\limits_{x\in B}d(t,x)$. It is evident that $B=B_{r_0}(t)$. Let us show that $\Phi(B)=B_{f(r_0)}(\Phi(t))\in \mathbf{B}_Y$. Suppose $y\in \Phi(B)$, then there exists $x\in B$ such that
\begin{equation}\label{e3.1}
y=\Phi(x).
\end{equation}

Consequently $d(x,t)\leqslant r_0$. According to equation~(\ref{e4.5}) and definition of $f$  we have the following relations
$$
\rho(\Phi(x), \Phi(t))=f(d(x,t))\leqslant f(r_0).
$$
The inequality $\rho(\Phi(x), \Phi(t))\leqslant f(r_0)$ and ~(\ref{e3.1}) imply that $y \in B_{f(r_0)}(\Phi(t))$.

Conversely let $y \in B_{f(r_0)}(\Phi(t))$. Since $\Phi$  is a bijection there exists $x\in X$ such that $y=\Phi(x)$. It is clear that $\rho(y,\Phi(t))\leqslant f(r_0)$. According to~(\ref{e4.5}) we have
$$
f(d(x,t))=\rho(\Phi(x),\Phi(t))\leqslant f(r_0).
$$
Since $f$ is strictly increasing we have $d(x,t)\leqslant r_0$ which is equivalent to $x\in B$. Consequently $y\in \Phi(B)$.
Analogously one can prove that for every $B\in \mathbf{B}_Y$ there exists $\bar{B}\in \mathbf{B}_X$ such that $\Phi(\bar{B})=B$.
The injectivity of~(\ref{iso}) follows from the fact that $\Phi$ is a bijection between $X$ and $Y$.
Thus mapping~(\ref{iso}) is a bijection between $\mathbf{B}_X$ and $\mathbf{B}_Y$.

To prove that mapping~(\ref{iso}) is an isomorphism between $\mathcal{H}(\mathbf{B}_X)$ and $\mathcal{H}(\mathbf{B}_Y)$ we have to establish the following equivalence
\begin{equation}\label{equiv}
\langle B_1,B_2\rangle \in A_{\mathcal H(\mathbf{B}_X)}\Leftrightarrow \langle \Phi(B_1),\Phi(B_2)\rangle \in A_{\mathcal H(\mathbf{B}_Y)}
\end{equation}
for all $B_1, B_2 \in \mathbf{B}_X$.
According to the definition of Hasse diagram the left part of~(\ref{equiv}) is equivalent to the conjunction of the following two conditions:
\begin{itemize}
  \item [A)] $B_1\subset B_2$,
  \item [B)] ($B_1\subseteq B \subseteq B_2) \Rightarrow (B_1= B \vee B_2=B)$.
\end{itemize}
Analogously, the right part of equivalence~(\ref{equiv}) is equivalent to the conjunction of conditions
\begin{itemize}
  \item [C)] $\Phi(B_1)\subset \Phi(B_2)$,
  \item [D)] ($\Phi(B_1)\subseteq \tilde{B} \subseteq \Phi(B_2)) \Rightarrow (\Phi(B_1)= \tilde{B} \vee \Phi(B_2)=\tilde{B})$.
\end{itemize}
The equivalence of conditions A) and C) follows directly from the fact that $\Phi$ is a bijection between $X$ and $Y$.
Suppose that the implication B)$\Rightarrow$D) does not hold. Consequently, there exists $\tilde{B}$ such that
$$
\Phi(B_1)\subseteq \tilde{B} \subseteq \Phi(B_2) \ \text{ and } \ \Phi(B_1)\neq \tilde{B} \neq \Phi(B_2).
$$
Since $\Phi$ is a bijection we have
$$
B_1\subseteq \Phi^{-1}(\tilde{B})\subseteq B_2 \ \text{ and } \   B_1 \neq \Phi^{-1}(\tilde{B}) \neq B_2,
$$
which contradicts condition B). The implication D)$\Rightarrow$B) can be proved analogously which completes the proof.
\end{proof}

Let $X$ and $Y$ be semimetric spaces. The mapping $F\colon X\to Y$ is ball-preserving if for every $Z\in \textbf{B}_X$ and $W\in \textbf{B}_Y$ the following relations hold
$$
F(Z)\in \textbf{B}_Y \,\text{ and } \, F^{-1}(W)\in \textbf{B}_X,
$$
where $F(Z)$ is the image of the set $Z$ under the mapping $F$ and $F^{-1}(W)$ is the preimage of the set $W$ under this mapping.

It was shown in~\cite{P(TIAMM)} that the representing rooted trees $T_X$ and $T_Y$ of finite ultrametric spaces $X$ and $Y$ are
isomorphic if and only if there exists a ball-preserving bijection $F\colon X \to Y$.

The proof of Theorem~\ref{t4} implies the following corollary.
\begin{corollary}\label{c4}
Let $X$ and $Y$ be finite semimetric spaces. Every weak similarity $\Phi\colon X\to Y$ is a ball-preserving bijection.
\end{corollary}

Recall that the \emph{indegree} $d_{G}(v)$ of a vertex $v$ in the directed graph $G$ is the number of arcs with the head $v$.
\begin{theorem}\label{t5}
Let $X$ and $Y$ be finite semimetric spaces. Then the Hasse diagrams $\mathcal H(\mathbf B_X)$ and $\mathcal H(\mathbf B_Y)$ are isomorphic as directed graphs if and only if there exists a bijective ball-preserving mapping $f\colon X\to Y$.
\end{theorem}

\begin{proof}
Suppose first that there exists $f\colon X\to Y$ which is bijective and ball-preserving. To prove that $\mathcal H(\mathbf B_X)\simeq \mathcal H(\mathbf B_Y)$ we have to establish the equivalence
$$
\langle B_1,B_2 \rangle \in A_{\mathcal H(\mathbf B_X)} \Leftrightarrow \langle f(B_1),f(B_2) \rangle \in A_{\mathcal H(\mathbf B_Y)}
$$
for all $B_1,B_2\in \mathbf B_X$. This may be proved similarly  to the proof of the second part of Theorem~\ref{t4}.

Conversely, suppose that $\mathcal H(\mathbf B_X)\simeq \mathcal H(\mathbf B_Y)$ with the isomorphism $F\colon \mathbf B_X\to\mathbf B_Y$. Denote by $B_X^1$ and $B_Y^1$ the sets of one-point balls of the spaces $X$ and $Y$. It is clear that the vertices of $\mathcal H(\mathbf B_X)$ and $\mathcal H(\mathbf B_Y)$ having zero indegree  coincide with $B_X^1$ and $B_Y^1$ respectively. Since $F$ is the isomorphism, the mapping
\begin{equation*}
f^*=F|_{B_X^1}\colon B_X^1\to B_Y^1
\end{equation*}
is a bijection. Define a mapping $f\colon X\to Y$ as follows
$$
f(x)=y \ \, \Leftrightarrow \ \, f^*(\{x\})=\{y\},
$$
which evidently is also a bijection. We claim that $f$ is ball-preserving.

It is clear that for $B\in \mathbf B_X$ the image $f(B)$ is  a ball in $Y$ if and only if the equality
\begin{equation}\label{eq11}
f^*\left(\bigcup\limits_{x\in B}\{\{x\}\}\right) = \bigcup\limits_{y \in B'} \{\{y\}\}
\end{equation}
holds  for some $B'\in \mathbf B_Y$.

Let $B$ be a ball in $X$. Denote by $G_B$ the directed subgraph of $\mathcal H(\mathbf B_X)$ induced by the set of all predecessors of $B\in V(\mathcal H(\mathbf B_X))$ and by $G_{F(B)}$ the directed subgraph of $\mathcal H(\mathbf B_Y)$ induced by the set of all predecessors of $F(B) \in V(\mathcal H(\mathbf B_Y))$. Since $\mathcal H(\mathbf B_X) \simeq \mathcal H(\mathbf B_Y)$ and $F$ is the isomorphism we have
\begin{equation}\label{eq12}
G_B\simeq G_{F(B)}.
\end{equation}

For all  $Z\in  V(\mathcal H(\mathbf B_X))$ ($W\in  V(\mathcal H(\mathbf B_Y))$)  define by $\Gamma_X(Z)$ ($\Gamma_X(W)$) the set of all predecessors of $Z\in V(\mathcal H(\mathbf B_Y))$  ($Z\in V(\mathcal H(\mathbf B_Y))$) having zero indegree.
By~(\ref{eq12}) we have $f^* (\Gamma_X(B)) = \Gamma_Y (F(B))$. To establish~(\ref{eq11}) it suffices to note that
$$
\Gamma_X(B)=\bigcup\limits_{x\in B}\{\{x\}\}  \ \text{ and  } \ \Gamma_Y(F(B))=\bigcup\limits_{y\in F(B)}\{\{y\}\}.
$$
These two equalities follow directly from the constructions of $\mathcal H(\mathbf B_X)$ and $\mathcal H(\mathbf B_Y)$ respectively.

The arguing for $f^{-1}$ is analogous.

\end{proof}

Theorems~\ref{t4} and~\ref{t5} imply the following.
\begin{corollary}\label{c55}
Let $X$ and $Y$ be finite semimetric spaces. If $X\we Y$ then there exists a bijective ball-preserving mapping $f\colon X\to Y$.
\end{corollary}

In~\cite[Theorem 2.8]{DPT(Howrigid)} it was proved that the Hasse diagram of a metric space $X$ is a tree if and only if $X$ is an ultrametric space. Hence the main theorem from~\cite{P(TIAMM)} may be obtained as a corollary of Theorem~\ref{t5}.
\begin{corollary}\label{c5}
Let $X$ and $Y$ be finite ultrametric spaces. Then the representing trees $T_X$ and $T_Y$ are isomorphic as rooted trees if and only if there exist a bijective ball-preserving mapping $f\colon X\to Y$.
\end{corollary}

\section{Acknowledgements}
The author is grateful to the anonymous referee for the numerous remarks, corrections, and suggestions which helped to improve the article. This publication is based on the research provided by the grant support of the State Fund For Fundamental Research (project F71/20570). The research was also partially supported by the Project 0117U006353 from the Department of Targeted Training of Taras Shevchenko National University of Kyiv at the NAS of Ukraine and by National Academy of Sciences of Ukraine within scientific research works for young scientists, Project 0117U006050.

\end{document}